\documentclass[12pt]{article}

\usepackage{amsthm}
\usepackage[english]{babel}
\usepackage[utf8]{inputenc}
\usepackage[bottom]{footmisc}
\usepackage{cite}
\usepackage{amssymb,amsfonts,amsmath}
\usepackage{mathrsfs}
\usepackage[compact]{titlesec}
\usepackage{verbatim}
\usepackage{indentfirst}
\usepackage{hyperref}
\hypersetup{
    colorlinks,
    citecolor=blue,
    filecolor=blue,
    linkcolor=blue,
    urlcolor=blue
}

\usepackage[left=2cm,right=2cm,top=3cm,bottom=3cm]{geometry}

\newtheorem*{theorem*}{Theorem}

\newtheorem{lem}{Lemma}
\newtheorem{cor}{Corollary}
\newtheorem{rem}{Remark}

\DeclareRobustCommand{\divby}{
  \mathrel{\vbox{\baselineskip.65ex\lineskiplimit0pt\hbox{.}\hbox{.}\hbox{.}}}
}

\newcommand\DistTo{\xrightarrow{
   \,\smash{\raisebox{-0.65ex}{\ensuremath{\scriptstyle\sim}}}\,}}

\author{T. Hakobyan, S. Vostokov}
\title{Honda formal group as Galois module in unramified extensions of local fields}	
\date{}

\begin{document}
\maketitle

 \begin{abstract}
    \small For given rational prime number $p$ consider the tower of finite extensions of fields  $K_0/\mathbb{Q}_p,$ $K/K_0, L/K, M/L$, where $K/K_0$ is unramified and $M/L$ is a Galois extension with Galois group $G$. Suppose one dimensional Honda formal group over the ring $\mathcal{O}_K$, relative to the extension $K/K_0$ and uniformizer $\pi\in K_0$ is given. The operation $x\underset{F}+y=F(x,y)$ sets a new structure of abelian group on the maximal ideal $\mathfrak{p}_M$ of the ring $\mathcal{O}_M$ which we will denote by $F(\mathfrak{p}_M)$. In this paper the structure of $F(\mathfrak{p}_M)$ as $\mathcal{O}_{K_0}[G]$-module is studied for specific unramified $p$-extensions $M/L$.%, provided that $W_F\cap F(\mathfrak{p}_L)=W_F\cap F(\mathfrak{p}_M)=W^s_F$ for certain  $s\geq 1$, where $W^s_F$ is the $\pi^s$-torsion, while $W_F=\bigcup_{n=1}^\infty W^n_F$ is the total $\pi$-torsion of a fixed algebraic closure $K^{alg}$ of $K$.
	
  \end{abstract}

\subsection{Introduction}

Let $p$ be a rational prime, $K/\mathbb{Q}_p, L/K, M/L$ be a tower of finite extensions of local fields, $M/L$ be a Galois extension with Galois group $G$ and $F$ be a one dimensional formal group law over the ring $\mathcal{O}_K$. The operation $x\underset{F}+y=F(x,y)$ sets a new structure of abelian group on the maximal ideal $\mathfrak{p}_M$ of the ring $\mathcal{O}_M$ which we will denote by $F(\mathfrak{p}_M)$. Taking into account the natural action of the group $G$ on $F(\mathfrak{p}_M)$, one may consider it as an $\text{End}_{\mathcal{O}_K}(F)[G]$-module, in which the multiplication by scalars from $\text{End}_{\mathcal{O}_K}(F)$  is performed by the rule  $f*x=f(x).$ We refer the reader  to \cite[Chapter 6, \S 3]{Cas}, \cite[Chapter 3, \S 6]{Neu}, \cite[Chapter 4]{Iwa3} and \cite[Chapter 4]{Sil} for more details concerning formal groups and the group $F(\mathfrak{p}_M)$. \newline
 If $F$ is a Lubin-Tate formal group law, then there is an injection $\mathcal{O}_K\hookrightarrow \text{End}_{\mathcal{O}_K}(F)$ (see \cite[Chapter 6, Prop. 3.3]{Cas}), which enables us to regard $F(\mathfrak{p}_M)$ as an $\mathcal{O}_K[G]$-module. The structure of this module in case of multiplicative formal group $F=G_m$ and $K=\mathbb{Q}_p$ is studied in sufficient detail in  \cite{Iwa4,Iwa5,Borevich}. The starting point of the current study is the following theorem of Borevich in \cite{Borevich}. 
\begin{theorem*} [Borevich, 1965]
Suppose $M/L$ is an unramified $p$-extension and $K=\mathbb{Q}_p$. If the fields $M$ and $L$ have the same irregularity degree \footnote{This means that the field $L$ contains a $p^s$-th primitive root of unity, while $M$ does not contain a primitive $p^{s+1}$-th root of unity for some $s\geq 1$} then for the $\mathcal{O}_K[G]$-module $U_M$ there exists a system of generating elements      $\theta_1,...,\theta_{n-1},\xi,\omega$ with the unique defining relation $\xi^{p^s}=\omega^{\sigma-1}$, where $n=[L:\mathbb{Q}_p]$ and $\sigma$ is a generating element of the Galois group $G=\mathrm{Gal}(M/L)$.
\end{theorem*}
It may seem that the group of principal units $E_M$ has nothing to do with formal groups, but in fact it is easy to show that for the multiplicative formal group $F=G_m$ there is an isomorphism $F(\mathfrak{p}_M)\cong E_M, x\mapsto 1+x$ of $\mathbb{Z}_p[G]$-modules.
The next stop in the course of investigations was the joint work of S.V.Vostokov and I.I.Nekrasov \cite{Nek7}, where they generalized the aforementioned theorem to the case of Lubin-Tate formal groups. More precisely, they managed to prove the following 
\begin{theorem*}[Vostokov-Nekrasov, 2014]
Suppose $M/L$ is an unramified $p$-extension and $F$ is a Lubin-Tate formal group for the prime element $\pi\in K$. Assume moreover that the fields $M$ and $L$ have the same irregularity degree, namely they contain a generator of $\ker [\pi^s]_F$ and do not contain a generator of $\ker [\pi^{s+1}]_F$ for some $s\geq 1$ \footnote{In fact, $\ker[\pi^s]_F$ is a cyclic $\mathcal{O}_K$-module, whenever $F$ is a Lubin-Tate formal group (See \cite[Chapter 3, Prop. 7.2]{Neu})}. Then for the $\mathcal{O}_K[G]$-module $F(\mathfrak{p}_M)$ there exists a system of generating elements $\theta_1,...,\theta_{n-1},\xi,\omega$ with the unique defining relation $[\pi^s]_F(\xi)=\omega^{\sigma}\underset{F}-\omega$, where $n=[L:K]$ and $\sigma$ is a generating element of the Galois group $G=\mathrm{Gal}(M/L)$. 
\end{theorem*}

The key point in this work was the proof of the triviality of the cohomology groups $H^{i}(G,F(\mathfrak{p}_M))$ $i=0,-1$ for unramified extensions $M/L$. In its turn, our work is devoted to the generalization of the last result to the case of Honda formal groups. Namely, let $K_0/\mathbb{Q}_p$ be a finite extension such that $K/K_0$ is unramified, $\pi\in K_0$ be a uniformizer, $F$ be a Honda formal group over $\mathcal{O}_K$ relative to the extension $K/K_0$ of type $u\in \mathcal{O}_{K,\varphi}[[T]]$. We refer the reader to \cite[\S\S2,3]{Hon} and \cite{Dem1} for more information concerning Honda formal groups. Suppose $K^{\mathrm{alg}}$ is a fixed algebraic closure of the field $K$, $\mathfrak{p}_{K^{\mathrm{alg}}}$ is the valuation ideal, i.e. the set of all points in $K^{\mathrm{alg}}$ with positive valuation. Define $W^n_F=\ker[\pi^n]_F\subset F(\mathfrak{p}_{K^{\mathrm{alg}}})$ to be the $\pi^n$-torsion submodule. More precisely, let $W^n_F=\{x\in \mathfrak{p}_{K^{\mathrm{alg}}}|[\pi^n]_F(x)=0\}$, where $[\pi^n]_F\in \text{End}_{\mathcal{O}_K}(F)$, and let $W_F=\bigcup\limits_{n=1}^{\infty} W^n_F$.\newline
It is known (See  \cite[\S 2, Thm. 3]{Hon}) that there is a ring embedding $\mathcal{O}_{K_0}\hookrightarrow \text{End}_{\mathcal{O}_K}(F)$, which allows as to regard $F(\mathfrak{p}_M)$ as an $\mathcal{O}_{K_0}[G]$-module. In this paper, using generators and defining relations we describe the structure of this module in the case of unramified $p$-extension $M/L$, provided that $W_F\cap F(\mathfrak{p}_L)=W_F\cap F(\mathfrak{p}_M)=W^s_F$, for certain $s\geq 1$. It is known that any finite unramified extension of a local field is a cyclic extension, so that  $G$ is a cyclic $p$-group.

\begin{flushleft}
\textit{We agree in the following notation}
\end{flushleft}

\flushleft $n-$the degree of the field $L$ over $K_0$;\\
$h-$the height of the type $u=\pi+\sum_{i\geq 1}a_iT^i$ of the formal goup $F$, i.e. the minimal $h$, for which $a_h$ is invertible. \\
$f-$the logarithm of $F$;\\
$p^m-$the order of the group $G=\text{Gal}(M/L)$;\\
$\sigma-$a generating element of $G$;\\
$\zeta_i, 1\leq i\leq h-$a fixed basis of the $\mathcal{O}_{K_0}/\pi^s\mathcal{O}_{K_0}$-module $W^s_F$;\\
$k_0,l-$the residue fields of $K_0$ and $L$ respectively;\\
$q-$the order of $k_0$;\\
$x\underset{F}+y:=F(x,y)$;\\
$\displaystyle\sum_{F;i=1}^k x_i:=x_1\underset{F}+x_2\underset{F}+...\underset{F}+x_k.$\\

\subsection{Auxiliary lemmas}

\begin{lem}\label{lem2.7}
 The $\mathcal{O}_{K_0}$-module $W^n_F$ is isomorphic to $(\mathcal{O}_{K_0}/\pi^n\mathcal{O}_{K_0})^h$.
\end{lem}

\begin{proof}
See \cite[Prop. 1]{Dem1}. 
\end{proof}

\begin{lem}\label{lem2.8}
In the case of an unramified extension $M/L$, the groups $H^i(G,F(\mathfrak{p}_M))$ are trivial for $i=0,-1$.
\end{lem}

\begin{lem}\label{lem2.9}
If the elements $x_1,x_2,..., x_k$ from $F(\mathfrak{p}_M)$ are such that the system $\{ N_{F(\mathfrak{p}_M)}(x_i), 1\leq i\leq k \}$ \footnote{ Here $N_{F(\mathfrak{p}_M)}$ is the $G$-module norm} is linearly independent in the $k_0$-vector space $F(\mathfrak{p}_M)/[\pi]_F(F(\mathfrak{p}_M)$, then so is the system $\{x^{\sigma^j}_i, 1\leq i\leq k,0\leq j\leq p^m-1\}$.

\end{lem}

\begin{lem}\label{lem2.10}
If the elements $x_1,x_2,..., x_k$ from $F(\mathfrak{p}_M)$ generate the $k_0$-vector space $F(\mathfrak{p}_M)/[\pi]_F(F(\mathfrak{p}_M)$, then they generate $F(\mathfrak{p}_M)$ as an $\mathcal{O}_{K_0}$-module.

\end{lem}

The proofs of lemmas 2.8-2.10 can be found in the article \cite{Nek7}, as well as in \cite[\S 3]{Borevich}.
\begin{lem}\label{lem2.11}
The natural linear map $$\varphi:F(\mathfrak{p}_L)/[\pi]_F (F(\mathfrak{p}_L))\rightarrow F(\mathfrak{p}_M)/[\pi]_F (F(\mathfrak{p}_M))$$ of $k_0$-vector spaces, induced by inclusion, has kernel of dimension $h$.

\end{lem}

\begin{proof}
Consider the elements $\eta_i=[\pi^{s-1}]_F\zeta_i, 1\leq i\leq h$. They form a basis of $W^1_F$ as an $\mathcal{O}_{K_0}/\pi\mathcal{O}_{K_0}$-module. Since $N_{F(\mathfrak{p}_M)}\eta_i=[p^m]_F\eta_i=0$, then by Lemma \ref{lem2.8} we get that $\eta_i=t^{\sigma}_i\underset{F}-t_i$ for some elements $t_i\in F(\mathfrak{p}_M)$. Suppose that $x\in F(\mathfrak{p}_L)$ and $x=[\pi]_F(y)$ for some  
$y\in F(\mathfrak{p}_M)$. Then $[\pi]_F(y^{\sigma}\underset{F}-y)=x^{\sigma}\underset{F}-x=0$, from which it follows that $$y^{\sigma}\underset{F}-y=\sum_{F;i=1}^h[a_i]_F(\eta_i)=\sum_{F;i=1}^h\left(([a_i]_F(t_i))^{\sigma}\underset{F}-{[a_i]_F(t_i)}\right),$$ for certain elements $a_i\in \mathcal{O}_{K_0}$, uniquely determined modulo $\pi$. The last relationship indicates the existence of $z\in F(\mathfrak{p}_L)$, for which 
$y=\displaystyle\sum_{F;i=1}^h[a_i]_F(t_i)\underset{F}+z$. Therefore, $$x=[\pi]_F(y)=\sum_{F;i=1}^h[a_i]_F([\pi]_F(t_i))\underset{F}+[\pi]_F(z).$$ Hence the elements $[\pi]_F(t_i),1\leq i\leq h$ constitute a basis of $\ker\varphi$. The lemma is proved.
\end{proof}

\begin{lem}\label{lem2.12}
The dimension of the $k_0$-vector space $F(\mathfrak{p}_L)/[\pi]_F (F(\mathfrak{p}_L))$ is equal to $n+h$.
\end{lem}

\begin{proof}
According to \cite[Chapter 4, Thm. 6.4]{Sil} for $i>\frac{e(L/\mathbb{Q}_p)}{p-1}$ there is an isomorphism of groups  $f:F(\mathfrak{p}^i_L)\DistTo \mathfrak{p}^i_L$, which is in fact an isomorphism of $\mathcal{O}_{K_0}$-modules due to the relation $f\circ [a]_F=af$ which holds for all $a\in \mathcal{O}_{K_0}$.  Consequently, $F(\mathfrak{p}^i_L)$ is a free $\mathcal{O}_{K_0}$- module of rank $n$. From the exactness of sequences of $\mathcal{O}_{K_0}$-modules: $$0\rightarrow F(\mathfrak{p}^{i+1}_L)\rightarrow F(\mathfrak{p}^i_L)\rightarrow l\rightarrow 0,\ i\geq 1$$ it follows that $F(\mathfrak{p}^i_L)$ is an $\mathcal{O}_{K_0}$-submodule of finite index in $F(\mathfrak{p}_L)$. Therefore $F(\mathfrak{p}_L)$ is a finitely generated $\mathcal{O}_{K_0}$-module of rank $n$. The theory of finitely generated modules over a PID yields $F(\mathfrak{p}_L)=T\oplus A$, where $T$ is the torsion submodule, which in our case coincides with $W^s_F$, while $A$ is a free $\mathcal{O}_{K_0}$-module of rank $n$.  In the long run, we get $$|F(\mathfrak{p}_L)/[\pi]_F (F(\mathfrak{p}_L))|=|T/[\pi]_F T|\cdot|A/[\pi]_F A|=q^h\cdot q^n=q^{n+h},$$ completing the proof of the lemma. 

\end{proof}

\begin{rem}\label{rem2.1}
Likewise we get that $\dim_{k_0}\left(F(\mathfrak{p}_M)/[\pi]_F (F(\mathfrak{p}_M))\right)=np^m+h$. 
\end{rem}

\begin{rem}
Since $F(\mathfrak{p}_M)$ is a finitely generated $\mathcal{O}_{K_0}$- module, then by Nakayama's lemma we obtain a new proof of the assertion of Lemma \ref{lem2.10}.
\end{rem}

\begin{lem}\label{lem2.13}
The elements $\zeta_i,1\leq i\leq h$ are linearly independent modulo $\ker\varphi.$

\end{lem}

\begin{proof}
Suppose the relation $\displaystyle\sum_{F;i=1}^h[a_i]_F\zeta_i=[\pi]_F(y)$ holds for some $a_i\in \mathcal{O}_{K_0}, y\in F(\mathfrak{p}_M)$. Applying the endomorphism $[\pi^s]_F$, we get that $[\pi^{s+1}]_F(y)=0$, which gives $[\pi^s]_F(y)=0$. The latter means that $\displaystyle\sum_{F;i=1}^h[\pi^{s-1}a_i]_F\zeta_i=0$, which is equivalent to the condition $a_i\divby \pi, 1\leq i\leq h$. The lemma is proved.

\end{proof}

\begin{cor}\label{cor2.1}
$h\leq n$
\end{cor}

\begin{proof}
In view of the lemmas proved, it follows that the maximal number of linearly independent vectors modulo $\ker{\varphi}$ in $F(\mathfrak{m}_L)/[\pi]_F (F(\mathfrak{p}_L))$ is equal to
$$\dim{\text{Im}\varphi}=\dim_{k_0}(F(\mathfrak{p}_L)/[\pi]_F (F(\mathfrak{p}_L)))-\dim{\ker{\varphi}}=(n+h)-h=n.$$  By Lemma \ref{lem2.13} we already have $h$ linearly independent vectors modulo $\ker\varphi$, from which the desired result follows. 
\end{proof}

\subsection{The main theorem}

\begin{theorem*}\label{hondamodule}
If the extension $M/L$ is unramified and $W_F\cap F(\mathfrak{p}_L)=W_F\cap F(\mathfrak{p}_M)=W^s_F$, for some $s\geq 1$, then $h\leq n$ and for the $\mathcal{O}_{K_0}[G]$-module $F(\mathfrak{p}_M)$ there exist a system of generating elements $\theta_j,\xi_i,\omega_i,1\leq j\leq n-h,1\leq i\leq h$ with the only defining relations $[\pi^s]_F(\xi_i)=\omega^{\sigma}_i\underset{F}-\omega_i,1\leq i\leq h$.
\end{theorem*}
\begin{proof}
 From the triviality of the group $H^0(G,F(\mathfrak{p}_M))$ follows the existence of elements $\xi_i\in F(\mathfrak{p}_M),1\leq i\leq h$, such that $N_{F(\mathfrak{p}_M)}(\xi_i)=\zeta_i$. Since $N_{F(\mathfrak{p}_M)}([\pi^s]_F(\xi_i))=[\pi^s]_F(\zeta_i)=0$ and the group  $H^{-1}(G,F(\mathfrak{p}_M))$ is trivial, there exist elements $\omega_i\in F(\mathfrak{p}_M),1\leq i\leq h$, satisfying the relations $[\pi^s]_F\xi_i=\omega^{\sigma}_i\underset{F}-\omega_i$. In view of Corollary \ref{cor2.1}, the system $\zeta_i,1\leq i\leq h$ can be supplemented to a basis modulo $\ker\varphi$ via elements $\varepsilon_j\in F(\mathfrak{m}_L),1\leq j\leq n-h$. For $1\leq j\leq n-h$ we select elements $\theta_j\in F(\mathfrak{p}_M)$, so that $N_{F(\mathfrak{p}_M)}(\theta_j)=\varepsilon_j$ for all $j$ and we prove that the system $$\mathscr{E}=\{\omega_i,\xi^{\sigma^k}_i,\theta^{\sigma^k}_j|1\leq i\leq h,1\leq j\leq n-h, 0\leq k\leq p^m-1\}$$ is linearly independent modulo $[\pi]_F (F(\mathfrak{p}_M))$. Assume the contrary that  there exist elements $a_i,a_{i,k},b_{j,k}\in \mathcal{O}_{K_0}$ and $\beta\in F(\mathfrak{p}_M)$ such that $$\sum_{F;i}[a_i]_F\omega_i\underset{F}+\sum_{F;i,k} [a_{i,k}]_F(\xi^{\sigma^k}_i)\underset{F}+\sum_{F;j,k}[b_{j,k}]_F(\theta^{\sigma^k}_j)\underset{F}+[\pi]_F(\beta)=0.$$  We apply $\sigma-1$ to both parts of the latter relation and use the relations $[\pi^s]_F\xi_i=\omega^{\sigma}_i\underset{F}-\omega_i,1\leq i\leq h$ to deduce the equality
$$\sum_{F;i,k}[a_{i,k}-a_{i,{k-1}}]_F(\xi^{\sigma^{k}}_i)\underset{F}+\sum_{F;j,k}[b_{j,k}-b_{j,{k-1}}]_F(\theta^{\sigma^{k}}_j)\underset{F}+\sum_{F;i}[a_i]_F[\pi^s]_F\xi_i\underset{F}+[\pi]_F(\beta^\sigma\underset{F}-\beta)=0.$$ From lemmas \ref{lem2.9} and \ref{lem2.13} it follows that the system $$\mathscr{E}_0=\{\xi^{\sigma^k}_i,\theta^{\sigma^k}_j|1\leq i \leq h,1\leq j\leq n-h, 0\leq k\leq p^m-1\}$$ is linearly independent modulo $[\pi]_F (F(\mathfrak{p}_M))$, so that $a_{i,k}(\text{mod}\ \pi)$ and $b_{j,k}(\text{mod}\ \pi)$ are independent of $k$. Therefore, without loss of generality we may assume that $$\sum_{F;i}[a_i]_F[\pi^s]_F\xi_i\underset{F}+[\pi]_F(\beta^\sigma\underset{F}-\beta)=0,$$ changing if needed $\beta$. From the obtained follows the existence of $b_i\in \mathcal{O}_{K_0},1\leq i\leq h$ such that $$\sum_{F;i}[a_i\pi^{s-1}]_F(\xi_i)\underset{F}+\beta^{\sigma}\underset{F}-\beta=\sum_{F;i}[b_i]_F\eta_i$$ Taking norms $N_{F(\mathfrak{p}_M)},$ the obtained relation leads to the equality $\sum_{F;i}[a_i\pi^{s-1}]_F(\zeta_i)=0$ which implies that $a_i\divby \pi,1\leq i\leq h$. From the linear independence of the system $\mathscr{E}_0$ it follows that $a_{i,k}\divby \pi$ and $b_{j,k}\divby \pi$ for all $i,j$ and $k$. This completes the proof of the linear independence of the system  $\mathscr{E}$. The number of vectors in it is $np^m+h=\dim_{k_0}\left(F(\mathfrak{p}_M)/[\pi]_F (F(\mathfrak{p}_M))\right)$, so that they generate the space $F(\mathfrak{p}_M)/[\pi]_F (F(\mathfrak{p}_M)).$ From lemma \ref{lem2.10} it follows that they generate $F(\mathfrak{p}_M) $ as an $\mathcal{O}_{K_0}$- module, and consequently the elements $\theta_j,\xi_i,\omega_i,1\leq j\leq n-h,1\leq i\leq h$ generate $F(\mathfrak{p}_M)$ as an $\mathcal{O}_{K_0}[G]$-module. It remains only to prove the assertion concerning defining relations. Let us further agree to write multiplication by elements of the ring $\mathcal{O}_{K_0}[G]$ through exponentiation. 
Suppose that the relation $$\sum_{F;i}\xi^{\alpha_i}_i\underset{F}+\sum_{F;i}\omega^{\beta_i}_i\underset{F}+\sum_{F;j}\theta^{\delta_j}_j=0,$$ holds for some elements $\alpha_i,\beta_i,\delta_j\in \mathcal{O}_{K_0}[G]$. Our goal is to prove the existence of elements $\gamma_i\in \mathcal{O}_{K_0}[G]$ for which $\alpha_i=\pi^s\gamma_i, \beta_i=(1-\sigma)\gamma_i$ and $\delta_j=0$. Indeed, let $\beta_i=b_i+(1-\sigma)\gamma_i$ for certain elements $b_i\in \mathcal{O}_{K_0}$ and $\gamma_i\in \mathcal{O}_{K_0}[G]$. Taking into account the relations $[\pi^s]_F\xi_i=\omega^{\sigma}_i\underset{F}-\omega_i$ for $1\leq i\leq h$ we get $$ \sum_{F;i}\omega^{b_i}_i\underset{F}+\sum_{F;i}\xi^{\alpha'_i}_i\underset{F}+\sum_{F;j}\theta^{\delta_j}_j=0,$$ where $\alpha'_i=\alpha_i-\pi^s\gamma_i$. Factoring the latter relation modulo $[\pi]_F (F(\mathfrak{p}_M))$ and recalling that the system $\mathscr{E}$ is a basis modulo $[\pi]_F (F(\mathfrak{p}_M))$, we find that there exist elements $b^{(1)}_i,\beta'_i,\delta^{(1)}_j\in \mathcal{O}_{K_0}[G]$ such that $b_i=\pi b^{(1)}_i,\alpha'_i=\pi\beta'_i,\delta_j=\pi\delta^{(1)}_j$. Therefore, for some elements $a_i\in \mathcal{O}_{K_0}$ we must have the equality $$ \sum_{F;i}\omega^{b^{(1)}_i}_i\underset{F}+\sum_{F;i}\xi^{\beta'_i-a_i\sum_k\sigma^k}_i\underset{F}+\sum_{F;j}\theta^{\delta^{(1)}_j}_j=0,$$ due to the fact that $\zeta_i=N_{F(\mathfrak{p}_M)}(\xi_i)=\xi^{\sum_k \sigma^k}_i$. For the same reasons, all $b^{(1)}_i$ and $\delta^{(1)}_j$ are divisible by $\pi$. By induction we construct sequences $(b^{(\nu)}_i)_{\nu\geq 0}$ and $(\delta^{(\nu)}_j)_{\nu\geq 0}$ satisfying the conditions $b^{(0)}_i=b_i,\delta^{(0)}_j=\delta_j, b^{(\nu)}_i=\pi b^{(\nu+1)}_i$ and $\delta^{(\nu)}_j=\pi\delta^{(\nu+1)}_j$ for all $\nu\geq 0,1\leq i\leq h,1\leq j\leq n-h$, from which it follows that $b_i=0$ for all $i$ and $\delta_j=0$ for all $j$. There remains only the relation $\sum_{F;i}\xi^{\alpha'_i}_i=0$. Let now $\alpha'_i=\sum_{k}a_{i,k}\sigma^k$, where $a_{i,k}\in \mathcal{O}_{K_0}$ for all $i,k$. The factorization modulo $[\pi]_F (F(\mathfrak{p}_M))$ yields $a_{i,k}=\pi b_{i,k}$. Further, we obtain that $$\sum_{F;i}\xi^{\sum_{k}b_{i,k}{\sigma^k}}_i=\sum_{F;i}[\lambda_i]_F(\zeta_i)=\sum_{F;i}\xi^{\lambda_i\sum_{k}\sigma^k}_i,$$ for some elements $\lambda_i\in \mathcal{O}_{K_0}$. Consequently $b_{i,k}(\text{mod} \pi)$ is the same for all $k$, and so on. In the end we get that $a_{i,k}=a_i$ and that $\sum_{F;i}[a_i]_F(\zeta_i)=0$, i.e. $a_i=\pi^s t_i$ for certain $t_i\in \mathcal{O}_{K_0}$ and therefore $$\alpha_i-\pi^s\gamma_i=\alpha'_i=\pi^s t_i\sum_{k}\sigma^k.$$ If we denote $\gamma'_i=\gamma_i+t_i\sum_{k}\sigma^k$, then we will have  $\alpha_i=\pi^s\gamma'_i$ and $\beta_i=(1-\sigma)\gamma_i=(1-\sigma)\gamma'_i$, thus completing the proof of the theorem .  

\end{proof}

\end{document}